\documentclass[11pt]{article}
\usepackage{amsmath, amssymb, amsthm}
\usepackage{verbatim}
\usepackage{multicol}
\usepackage{hyperref}
\usepackage{enumerate}
\usepackage{tikz}

\date{}
\title{\vspace{-1.8cm}Domination in 3-tournaments}
\author{
D\'aniel Kor\'andi \thanks{Department of Mathematics, ETH, 8092 Zurich. Email: daniel.korandi@math.ethz.ch.}
\and
Benny Sudakov \thanks{Department of Mathematics, ETH, 8092 Zurich.
Email: benjamin.sudakov@math.ethz.ch. 
Research supported in part by SNSF grant 200021-149111.}
}

\theoremstyle{plain}
\newtheorem{THM}{Theorem}
\newtheorem*{THM*}{Theorem}

\newtheorem{CONJ}[THM]{Conjecture}

\theoremstyle{definition}

\newcommand{\polylog}{\textrm{polylog}}

\begin{document}
\maketitle

\begin{abstract}
A 3-tournament is a complete 3-uniform hypergraph where each edge has a special vertex designated as its tail.
A vertex set $X$ dominates $T$ if every vertex not in $X$ is contained in an edge whose tail is in $X$. The domination number of $T$ is the minimum size of such an $X$. 
Generalizing well-known results about usual (graph) tournaments, Gy\'arf\'as conjectured that there are 3-tournaments with arbitrarily large domination number, and that this is not the case if any four vertices induce two triples with the same tail. In this short note we solve both problems, proving the first conjecture and refuting the second.
\end{abstract}

\medskip
A tournament is an oriented complete graph. The following generalization of tournaments to higher uniformity was suggested by Gy\'arf\'as. An $r$-tournament is a complete $r$-uniform hypergraph $T$ where each edge has a special vertex designated as its tail. We say that a vertex set $X$ dominates $T$ if every vertex outside $X$ is contained in a hyperedge whose tail is in $X$. The domination number of $T$ is the minimum size of such a dominating set $X$. Recently Gy\'arf\'as made the following two conjectures about 3-tournaments (see \cite{GY}).

\begin{CONJ}[Gy\'arf\'as]~
\begin{enumerate}
\item There are 3-tournaments with arbitrarily large domination number.
\item The domination number of a 3-tournament such that any four of its vertices induce at least two edges with the same tail is bounded by a constant.
\end{enumerate}
\end{CONJ}

These conjectures were motivated by analogous classic results about usual tournaments (see, e.g., \cite{MBOOK}). Indeed, it is well known that an $n$-vertex tournament can have a domination number as large as $(1+o(1))\log_2 n$, e.g., random tournaments have this property. On the other hand, if any three vertices of a tournament induce two edges with the same tail, i.e., there are no cyclic triangles, then the tournament is transitive and thus has a dominating set of size 1.

In this short note we construct 3-tournaments of arbitrarily large domination number such that any four vertices induce at least two edges with the same tail. This proves the first conjecture and disproves the second.

\medskip
The above conjectures turn out to be closely related to a problem about directed graphs. Recall that a directed graph has property $S_k$ if every set of size $k$ is dominated by some other vertex, i.e., for any set $X$ of size $k$, there is a vertex $v$ such that all $k$ edges between $v$ and $X$ exist and are directed towards $X$. The girth of a digraph is the minimum length of a directed cycle in it. Myers conjectured in 2003 \cite{M} that every digraph satisfying $S_2$ has girth bounded by an absolute constant.
A similar conjecture was later made in \cite{DMP}, motivated by algorithmic game theory. 
These conjectures were recently disproved by Anbalagan, Huang, Lovett, Norin, Vetta and Wu \cite{AHLNVW} (digraphs with property $S_2$ and girth four were constructed earlier
in \cite{BB}).  Their construction, which is based on a result of Haight \cite{H} (see also \cite{R}) in additive number theory, establishes the following.

\begin{THM}[\cite{AHLNVW}] \label{AHLNVW}
For any $k$ and $l$, there is a directed graph of girth at least $l$ that has property $S_k$.
\end{THM}

We will use this construction to resolve the above two problems about domination in 3-tournaments. Let $D$ be a digraph of girth at least 4 on a vertex set $V$, and fix an arbitrary ordering of $V$. We define $T_D$ to be a 3-tournament on the same set $V$ where the tail of each triple $A$ in $T_D$ is selected as follows.  
Look at all the directed paths in $D[A]$ of maximum length, and choose the tail of $A$ to be the smallest (according to the ordering we fixed) of the starting vertices.
Note that $D[A]$ is acyclic, so this tail has indegree 0 in $D[A]$. 
The following result together with Theorem \ref{AHLNVW} proves the existence of 3-tournaments with large domination number, and answers both questions of  Gy\'arf\'as.

\begin{THM}
If $D$ is a digraph of girth at least 4 with property $S_k$, then the tournament $T_D$ has domination number at least $k+1$. Furthermore, if $D$ has girth at least 5, then any four vertices in $T_D$ induce two triples sharing the same tail.
\end{THM}

\begin{proof}
Let $D$ be a digraph of girth at least 4 with property $S_k$. Suppose there is a set $X$ of size $k$ that dominates $T_D$. Then by property $S_k$, there is a vertex $v\in D$ such that all edges between $v$ and $X$ exist and are directed towards $X$. Since $X$ dominates $T_D$, and $v$ in particular, there is a triple $A$ containing $v$ whose tail is in $X$. But this tail has non-zero indegree in $D[A]$, contradicting the definition of $T_D$. So the domination number of $T_D$ is at least $k+1$.

Now suppose further that $D$ has girth at least 5, and pick an arbitrary set $B$ of four vertices. Then $D[B]$ is acyclic. Let $x\in B$ be the smallest among the starting vertices of the paths of maximum length in $D[B]$. If $D[B]$ is empty then $x$ is the tail of all three triples in $T_D[B]$ touching it. Otherwise, let $xy$ be the first edge of a path of maximum length in $D[B]$. Notice that there is no path of length 2 in $D[B]$ ending at $y$, as that would give a path longer than the one starting at $x$. 
But then $x$ is the tail of both triples in $B$ containing $x$ and $y$. Indeed, $z$ could only be the tail of $\{x,y,z\}$ if $zy$ was an edge in $D$ and $z$ was smaller than $x$ in the ordering, but that would contradict the choice of $x$.
\end{proof}

\medskip \noindent
\textbf{Remarks.}
\begin{itemize}
\item The questions of Gy\'arf\'as can also be asked for higher uniformity. The analogous construction (with a bit more complicated argument) shows that the domination number of $r$-tournaments can be arbitrarily large, even when any $r+1$ vertices induce $\lceil r/2\rceil$ hyperedges with the same tail. On the other hand, if any $r+1$ vertices induce $r$ edges sharing the same tail then it is not hard to see that there is a dominating vertex. It might be interesting to determine the minimum $i$ such that $i$ induced edges with the same tail contained in every subset of size $r+1$ imply a bounded domination number. We see that $\lceil r/2\rceil<i \le r$.
\item It is also natural to ask how large the domination number of a 3-tournament $T$ can be in terms of the number of vertices $n$. It is easy to show that domination number is always at most $\log_2 n$. Indeed, $T$ contains a vertex $v$ that is the tail of at least $\binom{n}{3}/n\ge n^2/7$ triples. Such a $v$ clearly dominates at least $n/2$ vertices. Applying induction on the remaining vertices and adding $v$ to the dominating set gives the above upper bound. On the other hand, the $k$ in our construction inherits a very weak dependence on $n$ from \cite{AHLNVW} and its number theoretic background. We only get a lower bound of $\polylog(\log^* n)$, where $\log^* n$ is the number of times one needs to iterate the logarithm function to reduce $n$ to a number $\le 1$, leaving a huge gap between the bounds.
\item The above-mentioned counterexample to the conjecture of Myers leads to another interesting question. How large can the girth be in an $n$-vertex digraph satisfying $S_2$? The lower bound from \cite{AHLNVW} has an order of magnitude $\polylog(\log^* n)$, while a logarithmic upper bound is easy to show. The following result of the second author together with Eyal Lubetzky and Asaf Shapira \cite{LS} gives an upper bound of $O(\log\log n)$. 
\begin{THM}
Let $D$ be an $n$-vertex digraph satisfying $S_2$, then its girth is at most $2\log_2\log_2 n$.
\end{THM}
\begin{proof}
Let $k=\log_2\log_2 n$ and let $D'$ be the $k$'th power of $D$, i.e., $xy$ is an edge in $D'$ if there is a directed path in $D$ from $x$ to $y$ of length at most $k$. It is easy to see that if a graph has property $S_a$, then its $b$'th power satisfies $S_{a^b}$. In particular, $D'$ has property $S_t$ with $t=2^k=\log_2 n$. If $D'$ has a cycle of length 2 then $D$ contains a cycle of length at most
$2\log_2\log_2 n$ and we are done. Otherwise, we can add directed edges to $D'$ to obtain an $n$-vertex tournament that, by monotonicity, will still satisfy $S_t$. However, it is well known (see \cite{MBOOK}) and easy to prove that such a tournament has more than $2^t=n$ vertices. This contradiction completes the proof.
\end{proof}
\end{itemize}

\end{document}